\documentclass[11pt, reqno]{amsart}

\usepackage{amsthm,amssymb,amstext,amscd,amsfonts,amsbsy,amsrefs,amsxtra,latexsym,amsmath,xcolor,mathrsfs,fancybox,upgreek, soul,url}
\usepackage[english]{babel}
\usepackage[all,cmtip]{xy}
\usepackage[latin1]{inputenc}
\usepackage{cancel}
\usepackage[draft]{hyperref}
\usepackage{comment}
\usepackage{mdframed}
\allowdisplaybreaks
\usepackage{mathtools}
\usepackage{enumerate}
\usepackage{thmtools}
\usepackage{thm-restate}

\usepackage{chngcntr}
\usepackage{enumitem}
\usepackage{etoolbox}
\usepackage{todonotes}
\usepackage{footmisc}
\usepackage{bigints}
\usepackage{enumerate}
\usepackage{graphicx}
\usepackage{youngtab}
\usepackage{bm}

\DeclarePairedDelimiter\abs{\lvert}{\rvert}%
\DeclarePairedDelimiter\norm{\lVert}{\rVert}%

\makeatletter
\let\oldabs\abs
\def\abs{\@ifstar{\oldabs}{\oldabs*}}
\let\oldnorm\norm
\def\norm{\@ifstar{\oldnorm}{\oldnorm*}}
\makeatother

\makeatletter
\newcommand*\bigcdot{\mathpalette\bigcdot@{.5}}
\newcommand*\bigcdot@[2]{\mathbin{\vcenter{\hbox{\scalebox{#2}{$\m@th#1\bullet$}}}}}
\makeatother

\newtheorem{theorem}{Theorem}
\newtheorem{lemma}[theorem]{Lemma}
\newtheorem{corollary}[theorem]{Corollary}
\newtheorem{proposition}[theorem]{Proposition}
\newtheorem{conjecture}[theorem]{Conjecture}

\theoremstyle{definition}

\theoremstyle{remark}

\newtheorem*{remark}{Remark}
\newtheorem*{example}{{\bf Example}}

\newtheorem*{remark*}{Remark}
\newtheorem*{remarks*}{Remarks}

\numberwithin{theorem}{section}
\numberwithin{proposition}{section}
\numberwithin{lemma}{section}
\numberwithin{corollary}{section}
\numberwithin{equation}{section}
\numberwithin{conjecture}{section}

\setlist[enumerate,1]{before=}
\AfterEndEnvironment{enumerate}{}

\newcommand{\N}{\mathbb{N}}
\newcommand{\Z}{\mathbb{Z}}

\newcommand{\Q}{\mathbb{Q}}
\newcommand{\SL}{{\text {\rm SL}}}

\renewcommand{\pmod}[1]{\  \,  \left( \mathrm{mod} \,  #1 \right)}
\newcommand{\tmod}[1]{\  \, ( \mathrm{mod} \,  #1 )}

\begin{document}

\title[On $t$-core and self-conjugate $(2t-1)$-core partitions in arithmetic progressions]{On $t$-core and self-conjugate $(2t-1)$-core partitions in arithmetic progressions}

\author{Kathrin Bringmann}
\address{Department of Mathematics and Computer Science, Division of Mathematics, University of Cologne, Weyertal 86-90, 50931 Cologne, Germany}
\email{kbringma@math.uni-koeln.de}

\author{Ben Kane}
\address{Department of Mathematics, University of Hong Kong, Pokfulam, Hong Kong}
\email{bkane@hku.hk}

\author{Joshua Males}
\address{Department of Mathematics and Computer Science, Division of Mathematics, University of Cologne, Weyertal 86-90, 50931 Cologne, Germany}
\email{jmales@math.uni-koeln.de}

\maketitle

\begin{abstract}
 We extend recent results of Ono and Raji, relating the number of self-conjugate $7$-core partitions to Hurwitz class numbers. Furthermore, we give a combinatorial explanation for the curious equality $2\operatorname{sc}_7(8n+1) = \operatorname{c}_4(7n+2)$.
 We also conjecture that an equality of this shape holds if and only if $t=4$, proving the cases $t\in\{2,3,5\}$ and giving partial results for $t>5$.
\end{abstract}

\section{Introduction and Statement of Results}

A \emph{partition} $\Lambda$ of $n\in\N$ is a non-increasing sequence $\Lambda \coloneqq (\lambda_1, \lambda_2, \dots, \lambda_s)$ of non-negative integers $\lambda_j$ such that $\sum_{1\leq j\leq s} \lambda_j = n$.
The \emph{Ferrers--Young diagram} of $\Lambda$ is the $s$-rowed diagram

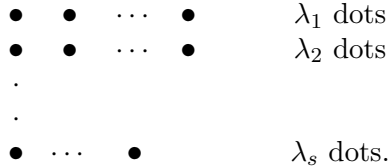
\begin{equation*}
\begin{matrix}
\text{\large $\bullet$} & \text{\large $\bullet$} & \cdots & \text{\large $\bullet$} & \qquad \text{ $\lambda_1$ dots}  \\
\text{\large $\bullet$} & \text{\large $\bullet$} & \cdots & \text{\large $\bullet$} & \qquad \text{ $\lambda_2$ dots}	\\
\cdot \\
\cdot \\
\text{\large $\bullet$} & \cdots & \text{\large $\bullet$} & {}  & \qquad \text{ $\lambda_s$ dots.}
\end{matrix}
\end{equation*}
We label the cells of the Ferrers--Young diagram as if it were a matrix, and let $\lambda_k'$ denote the number of dots in column $k$. The \emph{hook length} of the cell $(j,k)$ in the Ferrers--Young diagram of $\Lambda$ equals
\begin{equation*}
h(j,k) \coloneqq \lambda_j + \lambda_k' -k-j+1.
\end{equation*}
If no hook length in any cell of a partition $\Lambda$ is divisible by $t$, then $\Lambda$ is a \emph{$t$-core partition}. 
A partition $\Lambda$ is said to be \emph{self-conjugate} if it remains the same when rows and columns are switched.

\begin{example}
The partition $\Lambda = (3,2,1)$ of $6$ has the Ferrers--Young diagram
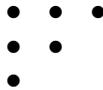
\begin{equation*}
\begin{matrix}
\text{\large $\bullet$} & \text{\large $\bullet$} & \text{\large $\bullet$} \\
\text{\large $\bullet$} & \text{\large $\bullet$} \\
\text{\large $\bullet$}
\end{matrix}
\end{equation*}
and has hook lengths 
$h(1,1)=5$, $h(1,2) = 3$, $h(1,3) = 1$, $h(2,1) = 3$, $h(2,2) = 1$, and $h(3,1) = 1$. Therefore, $\Lambda$ is a $t$-core partition for all $t \not\in\{ 1,3,5 \}$. Furthermore, switching rows and columns leaves $\Lambda$ unaltered, and so $\Lambda$ is self-conjugate.
\end{example}

 The theory of $t$-core partitions is intricately linked to various areas of number theory and beyond. For example, Garvan, Kim, and Stanton \cite{garvan1990cranks} used $t$-core partitions  to investigate special cases of the famous Ramanujan congruences for the partition function $p(n)$. Furthermore, $t$-core partitions encode the modular representation theory of symmetric groups $S_n$ and $A_n$ (see e.g. \cite{MR1321575,MR671655})

For $t,n \in \N$ we let $\operatorname{c}_t(n)$ denote the number of $t$-core partitions of $n$, along with $\operatorname{sc}_t(n)$ the number of self-conjugate $t$-core partitions of $n$. In 1997, Ono and Sze \cite{ono19974} investigated the relation between $4$-core partitions and class numbers. Denote by $H(|D|)$ ($D>0$ a discriminant) the $D$-th Hurwitz class number, which counts the number of $\SL_2(\Z)$-equivalence classes of integral binary quadratic forms of discriminant $D$, weighted by $\frac{1}{2}$ times the order of their automorphism group.\footnote{Some authors write $H(D)$ instead of $H(|D|)$; in particular this notation was used in \cite{OnoRaji}.} 
Then Ono and Sze proved the following theorem.
\begin{theorem}[Theorem 2 of \cite{ono19974}]\label{Theorem: Ono-Sze}
	If $8n+5$ is square-free, then 
	\begin{equation*}
	c_4(n) = \frac{1}{2} H(32n+20).
	\end{equation*}
\end{theorem}

 More recently Ono and Raji \cite{OnoRaji} showed similar relations between self-conjugate $7$-core partitions and class numbers. To state their result, let
 \begin{equation*}
 D_n \coloneqq \begin{cases}
 28n+56 & \text{ if } n \equiv 1 \pmod{4}, \\
 7n+14 & \text{ if } n \equiv 3 \pmod{4}.
 \end{cases}
 \end{equation*}
 \begin{theorem}[Theorem 1 of \cite{OnoRaji}]\label{Theorem: Ono-Raji}
 	Let $n \not\equiv -2 \pmod{7}$ be a positive odd integer. Then
 	\begin{equation*}
 	\operatorname{sc}_7(n) = \begin{cases}
 	\frac{1}{4} H(D_n) & \text{ if } n \equiv 1 \pmod{4}, \\
 	\frac{1}{2} H(D_n) & \text{ if } n \equiv 3 \pmod{8}, \\
 	0 & \text{ if } n \equiv 7 \pmod{8}.
 	\end{cases}
 	\end{equation*}
 \end{theorem}
 
 In particular, by combining Theorems \ref{Theorem: Ono-Sze} and \ref{Theorem: Ono-Raji} and using elementary congruence conditions, one may easily show that for $n \not\equiv 4 \pmod{7}$ and $56n+21$ square-free, 
\begin{equation}\label{Equation: 2sc7 = c4}
2 \operatorname{sc}_7(8n+1) = \operatorname{c}_4(7n+2).
\end{equation}

This fact hints at a deeper relationship between $\operatorname{sc}_{2t-1}$ and $\operatorname{c}_t$, which we investigate. Our main results pertain to the case of $t=4$. We begin by extending recent results of Ono and Raji \cite{OnoRaji}. Letting $\operatorname{sc}_7(n)$ denote the number of self-conjugate $7$-core partitions of $n$ and $(\frac{\cdot}{\cdot})$ denote the extended Jacobi Symbol, we may state our first theorem. For this, for $n\in\Q$ we set $H(n):=0$ if $n\notin\Z$ or $-n$ is not a discriminant.
\begin{theorem}\label{thm:sc7manyH}
	For every $n\in \N$, we have
	\[
	\operatorname{sc}_7(n)=\frac{1}{4}\left(H(28n+56)-H\left(\frac{4n+8}{7}\right)-2H(7n+14)+2H\left(\frac{n+2}{7}\right)\right).
	\]
\end{theorem} 
While Theorem \ref{thm:sc7manyH} gives a uniform formula for $\operatorname{sc}_7(n)$ as a linear combination of Hurwitz class numbers, it is also desirable to obtain a formula in terms of a single class number. For this, let $\ell\in\N_0$ be chosen maximally such that $n\equiv -2\tmod{2^{2\ell}}$ and extend the definition of $D_n$ to 
\begin{equation}\label{eqn:Dndef}
D_n:=\begin{cases}28n+56&\text{if }n\equiv 0,1\pmod{4},\\
7n+14&\text{if }n\equiv 3\pmod{4},\\
D_{\frac{n+2}{2^{2\ell}}-2}&\text{if }n\equiv 2\pmod{4},
\end{cases} 
\end{equation}
and 
\begin{equation}\label{eqn:nudef}
\nu_{n}:=\begin{cases}\frac{1}{4} &\text{if }n\equiv 0,1\pmod{4},\\ \frac{1}{2}&\text{if }n\equiv 3\pmod{8},\\ 
\nu_{\frac{n+2}{2^{2\ell}}-2}&\text{if }n\equiv 2\pmod{4},\\
0&\text{otherwise}. 
\end{cases}
\end{equation}

A binary quadratic form $[a,b,c]$ is called \begin{it}primitive\end{it} if $\gcd(a,b,c)=1$ and, for a prime $p$, \begin{it}$p$-primitive\end{it} if $p\nmid \gcd(a,b,c)$. We let $H_{p}(D)$ count the number of $p$-primitive classes of integral binary quadratic forms of discriminant $-D$, with the same weighting as $H(D)$. 
\begin{corollary}\label{cor:counting}
	For every $n\in\N$ we have
	\begin{equation*}
	\operatorname{sc}_7(n)=\nu_{n}H_7\left(D_n\right).
	\end{equation*}
\end{corollary}
\begin{remark*}
	For $n\not\equiv -2\pmod{7}$, one has 
	$H(D_n)=H_7(D_n)$ and hence
	the cases $n\equiv 1,3 \pmod{4}$
	of Corollary \ref{cor:counting}
	with $n\not\equiv -2 \pmod{7}$ are covered by Theorem \ref{Theorem: Ono-Raji}.
\end{remark*}

For $n+2$ squarefree, we may use Dirichlet's class number formula to obtain another representation for $\operatorname{sc}_7(n)$; Ono and Raji \cite[Corollary 2]{OnoRaji} covered
the case that $n\not\equiv -2\pmod{7}$ is odd.
\begin{corollary}\label{cor:Cor2}
	If $n\in\N$ is an integer for which $n+2$ is squarefree, then 
	\[
	\operatorname{sc}_7(n)=	-\frac{\nu_{n}}{D_n}
	\begin{cases}\vspace{3pt}
	\sum_{m=1}^{D_n-1}\left(\frac{-D_n}{m}\right) m&\text{if }n\not\equiv 
	-2
	\pmod{7},\\ \vspace{3pt}7^2\left(7+\left(\frac{\frac{D_n}{7^2}}{7}\right)\right)\sum_{m=1}^{\frac{D_n}{7^2}-1}\left(\frac{-\frac{D_n}{7^2}}{m}\right) m&\text{if }n\equiv 
	-2
	\pmod{7}.
	\end{cases}
	\]
\end{corollary}

The following corollary relates $\operatorname{sc}_7(m)$ with $m+2$ not necessarily squarefree
to $\operatorname{sc}_7(n)$ with $n+2$ squarefree, for which Corollary \ref{cor:Cor2} applies. 
The cases $\ell=r=0$ with $n\not\equiv 
-2
\pmod{7}$ odd were proven in \cite[Corollary 3]{OnoRaji}.
For this $\mu$ denotes the M\"{o}bius function and $\sigma(n):=\sum_{d\mid n} d$.

\begin{corollary}\label{cor:Cor3}
	If $n\in\N$ satisfies $n+2$ squarefree, $\ell,r\in\N_0$, and $f\in\N$ with $\gcd(f,14)=1$, then
	\[
	\operatorname{sc}_7\left((n+2)
	2^{2\ell}f^27^{2r}	-2\right)=7^r\operatorname{sc}_7(n)\sum_{d|f} \mu(d)\left(\frac{-D_n}{d}\right)\sigma \left(\frac{f}{d}\right).
	\]
\end{corollary}

We also provide a combinatorial explanation for Corollary \ref{cor:counting}. To do so, we first extend techniques of Ono and Sze \cite{ono19974} and explicitly describe the possible abaci (defined in Section \ref{Section: combinatorics}) of self-conjugate $7$-core partitions . Then, in \eqref{Equation: defnintion phi} below we construct an explicit map $\phi$ sending self-conjugate $7$-core partitions to binary quadratic forms, via abaci and extended $t$-residue diagrams (defined in Section \ref{Section: combinatorics}).

 In order to describe the image of this map, for a prime $p$ and a discriminant $D=\Delta f^2$ with $\Delta$ fundamental, we call a binary quadratic form of discriminant $D$ \begin{it}$p$-totally imprimitive\end{it} if the power of $p$ dividing $\gcd(a,b,c)$ equals the power of $p$ dividing $f$ (i.e., if the power of $p$ dividing $\gcd(a,b,c)$ is maximal). Furthermore, recall that two binary quadratic forms of discriminant $D$ are said to be in the same {\it genus} if they represent the same values in $(\Z/D\Z)^*$. The {\it principal binary quadratic form} of discriminant $D$, which acts as the identity under Gauss's composition laws, is defined by $u^2+Dv^2 \text{ if } D \equiv 0 \pmod{4}$ and $u^2+uv+\frac{D-1}{4}v^2 \text{ if } D \equiv 1 \pmod{4}$. We call the genus containing the principal binary quadratic form of discriminant $D$ the {\it principal genus}. The image of $\phi$ is then described in the following theorem.
\begin{theorem}\label{Theorem: main - quad forms}
For every $n\in\N$, the image of $\phi$ is a unique non-principal genus of $7$-primitive and $2$-totally imprimitive binary quadratic forms with discriminant $-28n-56$.
 Moreover, suppose that $\ell$ is chosen maximally such that $n\equiv -2\tmod{2^{2\ell}}$ and $\frac{7n+14}{2^{2\ell}}$ has $r$ distinct prime divisors. Then every equivalence class in this genus is the image of $\nu_n 2^r$ many self-conjugate $7$-cores of $n$.
\end{theorem}

Note that Theorem \ref{Theorem: main - quad forms} along with \cite[Theorem 6]{ono19974} provides a combinatorial explanation for \eqref{Equation: 2sc7 = c4}.
The cases $t\in\{2,3\}$ are simple to describe, and immediately imply that relationships similar to \eqref{Equation: 2sc7 = c4} along arithmetic progressions do not exist for $t\in\{2,3\}$, which we see in Section \ref{Sec: t=2,3}. We prove a similar result for $t=5$ in Proposition \ref{Prop: t=5}. Based on these results we offer the following conjecture, along with partial results on possible values of $t \pmod{6}$ along with the possible shapes of arithmetic progressions in Section \ref{Section: t>5}.
\begin{conjecture}\label{Theorem: main intro}
	The only occurrence of arithmetic progressions for which $\operatorname{c}_t$ and $\operatorname{sc}_{2t-1}$ agree up to integer multiples non-trivially (even asymptotically) is when $t=4$. 
\end{conjecture}

The paper is organised as follows. In Section \ref{sec:manyone}, we provide proofs for Theorem \ref{thm:sc7manyH} and Corollary \ref{cor:Cor2}, Corollaries \ref{cor:counting} and \ref{cor:Cor3} are shown in Section \ref{sec:squarepartcounting}.  Section \ref{Section: combinatorics} is dedicated to providing a combinatorial explanation of Theorem \ref{Theorem: Ono-Raji} and its generalization in Corollary \ref{cor:counting}. In Section \ref{Section: other t and conjecture} we prove Conjecture \ref{Theorem: main intro} in the cases $t\in\{2,3,5\}$ and provides partial results for larger $t$.

\section*{Acknowledgments}
The research of the first author is supported by the Alfried Krupp Prize for Young University Teachers of the Krupp foundation and has received funding from the European Research Council (ERC) under the European Union's Horizon 2020 research and innovation programme (grant agreement No. 101001179).
The research of the second author was supported by grants from the Research Grants Council of the Hong Kong SAR, China (project numbers HKU 17301317 and 17303618). 
The authors thank Ken Ono for insightful discussions pertaining to this paper, and for hosting the third author at the University of Virginia during which this research was partially completed. The authors also thank Andreas Mono for useful comments on an earlier draft. The authors also thank the referees for helpful comments.

\section{Proofs of Theorem \ref{thm:sc7manyH} and Corollary \ref{cor:Cor2}}\label{sec:manyone}
Our investigation for the case $t=4$ begins by packaging the number of self-conjugate $7$-cores into a generating function and using the fact that it is a modular form to relate $\operatorname{sc}_7(n)$ to class numbers. We thus define
\[
S(\tau):=\sum_{n\geq 0} \operatorname{sc}_7(n) q^{n+2}.
\]
As stated on \cite[page 4]{OnoRaji}, $S$ is a modular form of weight $\frac{3}{2}$ on $\Gamma_0(28)$ with character $(\frac{28}{\cdot})$.
\subsection{Proof of Theorem \ref{thm:sc7manyH}}
To prove Theorem \ref{thm:sc7manyH}, we let
\[
\mathcal{H}_{\ell_1,\ell_2}(\tau):=\mathcal{H}\big|(U_{\ell_1,\ell_2}-\ell_2U_{\ell_1}V_{\ell_2})(\tau).
\]
Here for $f(\tau):=\sum_{n\in\Z} c_f(n)q^n$
\[
f\big|U_d(\tau):=\sum_{n\in\Z} c_f(dn)q^n,\qquad f\big|V_d(\tau):=\sum_{n\in\Z} c_f(n)q^{dn},
\]
and
\begin{equation*}
\mathcal{H}(\tau):=\sum_{\substack{D\geq 0\\ D\equiv 0,3\pmod 4}}H(D)q^D.
\end{equation*}
\begin{proof}[Proof of Theorem \ref{thm:sc7manyH}]
	Shifting $n\mapsto n-2$ in Theorem \ref{thm:sc7manyH} and
	taking the generating function of both sides, the claim of the theorem is equivalent to
	\begin{equation}\label{eqn:toprove}
	S=\frac{1}{4}\mathcal{H}_{1,2}\big|\left(U_{14}-U_2\big|V_7\right).
	\end{equation}

	By \cite[Lemma 2.3 and Lemma 2.6]{BKPeterssonQF}, both sides of \eqref{eqn:toprove} are modular forms of weight $\frac{3}{2}$ on $\Gamma_0(56)$ with character $
	(\frac{28}{\cdot})
	$. By the valence formula, it thus suffices to check \eqref{eqn:toprove} for the first $12$ coefficients; this has been done by computer, yielding \eqref{eqn:toprove} and hence Theorem \ref{thm:sc7manyH}. 
\end{proof}

\subsection{Rewriting $\operatorname{sc}_7(n)$ in terms of representation numbers}
The next lemma rewrites $\operatorname{sc}_7(n)$ in terms of the representation numbers ($m\in \N_0$)
\[
r_{3}(m):=\#\left\{ \bm{x}\in \Z^3: x_1^2+x_2^2+x_3^2=m\right\}.
\]
For $m\in\Q\setminus\N_0$, we furthermore set $r_3(m):=0$ for ease of notation.
\begin{lemma}\label{lem:sc7Theta^3}
	\hfill
	\begin{enumerate}[leftmargin=20pt]
		\item[\rm (1)] For $n\in \N$, we have
		\begin{equation*}
		\operatorname{sc}_7(n)=\frac{1}{48}\left(r_3(7n+14)-r_3\left(\frac{n+2}{7}\right)\right). 
		\end{equation*}
		
		\item[\rm (2)] 	If $n\equiv 
		-2
		\pmod{7}$, then we have 
		\begin{equation*}
		\operatorname{sc}_7(n)=\frac{1}{48}\left(\left(7+\left(\frac{\frac{D_n}{7^2}}{7}\right)\right)r_3\left(\frac{n+2}{7}\right)-7r_3\left(\frac{n+2}{7^3}\right)\right). 
		\end{equation*}
	\end{enumerate}
\end{lemma}
\begin{proof}
	(1)\;
	 By the proof of \cite[Lemma 4.1]{BKPeterssonQF} we have 
	\[
	\Theta^3(\tau)=\sum_{n\geqslant0} r_3(n) q^n=12 \mathcal{H}_{1,2}\big|U_2(\tau),
	\]
	where  $\Theta(\tau):=\sum_{n\in\Z} q^{n^2}$ is the usual theta function.
	Plugging this into \eqref{eqn:toprove}, the claim follows after picking off the Fourier coefficients and shifting $n\mapsto n+2$.
	\\
	(2)\;
	Recall that 
	for $f(\tau) =\sum_{n\in\Z} c_f(n) q^n$ a modular form of weight $\lambda +\frac{1}{2}\in\Z+\frac 12$, the \emph{$p^2$-th Hecke operator} is defined as
	\begin{equation*}
	f| T_{p^2}(\tau) = \sum_{n\geqslant 0}\left(c_f\left(p^2n\right) +\left(\frac{(-1)^{\lambda}n}{p}\right)p^{\lambda  -1}c_f(n)+p^{2\lambda-1}c_f\left(\frac{n}{p^2}\right)\right)q^n.
	\end{equation*}	
	It is well-known that
	\begin{equation}\label{Hecke1}
	\Theta^3|T_{p^2}=(p+1)\Theta^3.
	\end{equation}
	Rearranging \eqref{Hecke1} and comparing coefficients we obtain, by \eqref{Hecke1},	for $m:=n+2\equiv 0\pmod{7}$,
	\[
	r_3(7m)= 8r_3\left(\frac{m}{7}\right)-\left(\frac{-\frac{m}{7}}{7}\right) r_3\left(\frac{m}{7}\right) - 7r_3\left(\frac{m}{7^3}\right). 
	\]
	The claim follows by (1).
\end{proof}

\subsection{Formulas in terms of single class numbers}
We next turn to formulas for $\operatorname{sc}_7(n)$ in terms of a single class number. 
\begin{corollary}\label{cor:sc7oneH}
	\hfill
	\begin{enumerate}[leftmargin=*,label=\rm(\arabic*)]
		\item For $n\not\equiv -2\pmod{7}$
		and $n\not\equiv 2\pmod{4}$, we have 
		\begin{equation*}
		\operatorname{sc}_7(n)=\nu_{n}H(D_n).
		\end{equation*}
		
		\item For $n\equiv -2\pmod{7}$, $n\not\equiv -2\pmod{7^3}$, and $n\not\equiv 2\pmod{4}$, we have 
		\begin{equation*}
		\operatorname{sc}_7(n)=\left(7+\left(\frac{\frac{D_n}{7^2}}{7}\right)\right)\nu_{n}H\left(\frac{D_n}{7^2}\right).
		\end{equation*}
		\item
		If $n\equiv 2\pmod{4}$, then
		\[
		\operatorname{sc}_7(n)=\operatorname{sc}_7\left(\frac{n+2}{4}-2\right).
		\]
		\item
		If $n\equiv -2\pmod{7^2}$, then 
		\[
		\operatorname{sc}_7(n)=7\operatorname{sc}_7\left(\frac{n+2}{7^2}-2\right).
		\]
	\end{enumerate}
	
\end{corollary}
\begin{remark*}
	For $n\not\equiv 2\pmod{4}$, we have $7(n+2)\mid D_n$, so $n\equiv -2\pmod 7$ implies that $7^2\mid D_n$, and hence Corollary \ref{cor:sc7oneH} (2) is meaningful.
\end{remark*}

\begin{proof}[Proof of Corollary \ref{cor:sc7oneH}]
	(1)\; Since $n\not\equiv -2\pmod{7}$, the final term in Lemma \ref{lem:sc7Theta^3} (1) vanishes, giving 
	\begin{equation*}
	\operatorname{sc}_7(n)= \frac{1}{48}r_3(7n+14).
	\end{equation*}
	The claim then follows immediately by plugging in
	the well-known formula of Gauss (see e.g. \cite[Theorem 8.5]{OnoBook})
	\begin{equation}\label{eqn:r3Gauss}
	r_3(n)=\begin{cases} 12 H(4n)&\text{if }n\equiv 1,2\pmod{4},\\
	24H(n)&\text{if }n\equiv 3\pmod{8},\\
	r_3\left(\frac{n}{4}\right)&\text{if }4\mid n,\\
	0&\text{otherwise}.
	\end{cases}
	\end{equation}
	\\
	\noindent
	(2)\; Since	$7^3\nmid (n+2)$, the final term in Lemma \ref{lem:sc7Theta^3} (2) vanishes, giving 
	\[
	\operatorname{sc}_7(n)= 
	\frac{1}{48}\left(7+\left(\frac{\frac{D_n}{7^2}}{7}\right)\right) r_3\left(\frac{n+2}{7}\right).
	\]
	\rm
	The claim then immediately follows by plugging in \eqref{eqn:r3Gauss}. 
	\\
	\noindent
	(3)\; 
	Since $n\equiv 2\pmod{4}$, we have $4\mid (n+2)$, and hence \eqref{eqn:r3Gauss} and Lemma \ref{lem:sc7Theta^3} (1) imply the claim.
	
	\noindent
	(4)\;
	Since $n\equiv -2\pmod{7^2}$, $7^3\mid D_n$,
	so  $7\mid \frac{D_n}{7^2}$.
	Hence Lemma \ref{lem:sc7Theta^3} (1), (2) imply the claim.
\end{proof}

\subsection{Proof of Corollary \ref{cor:Cor2}}
We next consider the special case that $n+2$ is squarefree and use Dirichlet's class number formula to obtain another formula for $\operatorname{sc}_7(n)$.
\begin{proof}[Proof of Corollary \ref{cor:Cor2}]
	Note that since $n+2$ is squarefree, either $-D_n$ is fundamental (for $n\not\equiv -2\pmod{7}$) or $-\frac{D_n}{7^2}$ is fundamental (for $n\equiv -2\pmod{7}$).
	Dirichlet's class number formula (see e.g. \cite[Satz 3]{ZagierZeta}) states that
	\begin{equation}\label{eqn:ClassNumberFormulaCorrect}
	H(|D|)= -\frac{1}{|D|} \sum_{m=1}^{|D|-1} \left(\frac{D}{m}\right)m.
	\end{equation}
	By Corollary \ref{cor:sc7oneH} (1),  (2) (the conditions given there are satisfied because $n+2$ is squarefree and thus neither $n\equiv 2 \pmod 4$ nor $n\equiv-2\tmod{7^3}$), we have 
	\begin{equation}\label{SH}\operatorname{sc}_7(n)=\nu_{n}
	\begin{cases}
	H(D_n)  & \text{if } n\not\equiv -2\pmod{7},\\
	\left(7+\left(\frac{\frac{D_n}{7^2}}{7}\right)\right)	H\left(\frac{D_n}{7^2}\right)  & \text{if } n\equiv 
	-2
	\pmod{7}.
	\end{cases}
	\end{equation}
	Since $-D_n$ is fundamental in the first case and $-\frac{D_n}{7^2}$ is fundamental in the second case, we may plug in \eqref{eqn:ClassNumberFormulaCorrect} with $D= -D_n$ in the first case and $D=-\frac{D_n}{7^2}$ in the second case.
	
	Thus for $n\not\equiv -2 \pmod{7}$ we plug 
	\[
	H\left(D_n\right)=-\frac{1}{D_n} \sum_{m=1}^{D_n-1} \left(\frac{-D_n}{m}\right)m
	\]
	into \eqref{SH}, while for $n\equiv -2 \pmod{7}$ we plug in
	\[
	H\left(\frac{D_n}{7^2}\right)=-\frac{7^2}{D_n} \sum_{m=1}^{\frac{D_n}{7^2}-1}\left(\frac{-\frac{D_n}{7^2}}{m}\right)m.
	\]
	This yields the claim. 
\end{proof}

\section{Proofs of Corollaries \ref{cor:counting} and \ref{cor:Cor3}}\label{sec:squarepartcounting}
This section relates $\operatorname{sc}_7(m)$ and $\operatorname{sc}_7(n)$ if $\frac{m+2}{n+2}$ is a square.
\subsection{A recursion for $\operatorname{sc}_7(n)$}
In this subsection, we consider the case 	$\frac{m+2}{n+2}=2^{2j}7^{2\ell}$. 
\begin{lemma}\label{lem:div4}
	Let $\ell\in\N_0$ and $n\in\N$.
	\begin{enumerate}[leftmargin=20pt]
		\item[\rm (1)] We have 
		\begin{equation*}
		\operatorname{sc}_7\left((n+2)2^{2\ell}-2\right)=\operatorname{sc}_7(n).
		\end{equation*}
		
		\item[\rm (2)] We have 
		\[
		\operatorname{sc}_7\left((n+2)7^{2\ell}-2\right)=7^\ell\operatorname{sc}_7(n).
		\]
	\end{enumerate}
\end{lemma}
\begin{proof}
	(1)\; 
	Corollary \ref{cor:sc7oneH} (3) gives inductively that for $0\leq j\leq \ell$ we have 
	\[
	\operatorname{sc}_7\left((n+2)2^{2\ell}-2\right) =\operatorname{sc}_7\left((n+2)2^{2(\ell-j)}-2\right).
	\]
	In particular, $j=\ell$ yields the claim.
	\\
	(2)\;
	The claim is trivial if $\ell=0$. For $\ell\geq 1$, Corollary \ref{cor:sc7oneH} (4) inductively yields that for $0\leq j\leq \ell$
	\[
	\operatorname{sc}_7\left((n+2)7^{2\ell}-2\right)=7^j\operatorname{sc}_7\left((n+2)7^{2(\ell-j)}-2\right).
	\]
	The case $j=\ell$ is precisely the claim.\qedhere
\end{proof}

\subsection{Proof of Corollary \ref{cor:Cor3}}
We are now ready to prove Corollary \ref{cor:Cor3}. 
\begin{proof}[Proof of Corollary \ref{cor:Cor3}]
	We first use Lemma \ref{lem:div4} (1), (2) to obtain that
	\begin{equation}\label{eqn:2and7gone}
	\operatorname{sc}_7\left((n+2)
	2^{2\ell}
	f^27^{2r} -2\right)
	=7^r\operatorname{sc}_7\left((n+2)f^2-2\right).
	\end{equation}

	We split into the case $n\not\equiv 
	-2
	\pmod{7}$ (in which case $-D_n$ is fundamental) and $n\equiv 
	-2
	\pmod{7}$ (in which case $-\frac{D_n}{7^2}$ is fundamental).

	First suppose that $n\not\equiv 
	-2
	\pmod{7}$. We use Corollary \ref{cor:sc7oneH} (1) to obtain 
	\begin{equation*}
	\operatorname{sc}_7\left((n+2)f^2-2\right)=\nu_{n}H\left(D_nf^2\right)
	\end{equation*}
	We then plug in \cite[p. 273]{Cohen} ($-D$ a fundamental discriminant) 
	\begin{equation}\label{eqn:Df^2}
	H\left(Df^2\right)=H(D)\sum_{1\leq d|f} \mu(d)\left(\frac{-D}{d}\right)\sigma\left(\frac{f}{d}\right).
	\end{equation}
	Hence by Corollary \ref{cor:sc7oneH} (1)
	\[
	\operatorname{sc}_7\left((n+2)f^2-2\right)=\operatorname{sc}_7(n)\sum_{1\leq d|f} \mu(d)\left(\frac{-D_n}{d}\right)\sigma\left(\frac{f}{d}\right),
	\]
	and plugging back into \eqref{eqn:2and7gone} yields the corollary in that case.

	We next suppose that $n\equiv -2\pmod 7$. 
	First note that since $7\nmid f$ and $n+2$ is squarefree, $(n+2)f^2-2\not\equiv -2\pmod{7^3}$ 
	and $n\not\equiv 2\pmod{4}$.
	We plug in Corollary \ref{cor:sc7oneH} (2), use \eqref{eqn:Df^2} 
	(recall that $-\frac{D_n}{7^2}$ is fundamental),
	and note that $\left(\frac{\frac{D_n f^2}{7^2}}{7}\right) = \left(\frac{\frac{D_n}{7^2}}{7}\right)$ to obtain that
	\[
	\operatorname{sc}_7\left((n+2)f^2-2\right)= \left(7+\left(\frac{\frac{D_n}{7^2}}{7}\right)\right)\nu_{n}H\left(\frac{D_n}{7^2}\right)\sum_{1\leq d|f} \mu(d)\left(\frac{-\frac{D_n}{7^2}}{d}\right)\sigma\left(\frac{f}{d}\right).
	\]
	We then use Corollary \ref{cor:sc7oneH} (2) again and plug back into \eqref{eqn:2and7gone} to conclude that
	\[
	\operatorname{sc}_7\left((n+2)
	2^{2\ell}f^27^{2r}
	-2\right)=7^r\operatorname{sc}_7(n)\sum_{1\leq d|f} \mu(d)\left(\frac{-\frac{D_n}{7^2}}{d}\right)\sigma\left(\frac{f}{d}\right).
	\]
	Since $7\nmid f$, we have $\left(\frac{-\frac{D_n}{7^2}}{d}\right)=\left(\frac{-D_n}{d}\right)$ for $d\mid f$. Therefore the corollary follows.
\end{proof}

\subsection{Proof of Corollary \ref{cor:counting}}
We next rewrite Corollary \ref{cor:sc7oneH} (2) in order to uniformly package Corollary \ref{cor:sc7oneH} (1), (2), and (3).
We first require a lemma relating the $7$-primitive class numbers $H_7$ and the Hurwitz class numbers.
\begin{lemma}\label{lem:H7diff}
	For a discriminant $-D$, we have 
	\[
	H_7(D)=H(D)-H\left(\frac{D}{7^2}\right).
	\]
\end{lemma}
\begin{proof}
A simple manipulation using the inclusion-exclusion principle immediately yields the claimed formula.
	\end{proof}
	
	To finish the proof of Corollary \ref{cor:counting}, for a fundamental discriminant $-\Delta$, we also require the evaluation of
	\[
	C_{r,\Delta}:=\sum_{d\mid 7^r} \mu(d)\left(\frac{-\Delta}{d}\right)\sigma\left(\frac{7^r}{d}\right)-\sum_{d\mid 7^{r-1}} \mu(d)\left(\frac{-\Delta}{d}\right)\sigma\left(\frac{7^{r-1}}{d}\right).
	\]
	A straightforward calculation gives 
	the following lemma.
	\begin{lemma}\label{lem:Creval}
		For $r\in\N$ we have 
		\[
		C_{r,\Delta}=7^{r-1}\left(7+\left(\frac{\Delta}{7}\right)\right).
		\]
	\end{lemma}
	We are now ready to prove Corollary \ref{cor:counting}.
	\begin{proof}[Proof of Corollary \ref{cor:counting}]
		We first consider the case that $n\not\equiv 2\pmod{4}$. If $n\not \equiv -2\pmod{7}$, {then Corollary} \ref{cor:counting} follows directly from Corollary \ref{cor:sc7oneH} (1) and  Lemma \ref{lem:H7diff}.
		
		For $n\equiv -2\pmod{7}$, we choose $r_n\in\N_0$ maximally such that $n\equiv -2\pmod{7^{2r_n+1}}$ and proceed by induction on $r_n$. For $r_n=0$ we have $D_n=\Delta_n f^27^2$ with $-\Delta_n$ a fundamental discriminant and $7\nmid f$. Since $7\nmid f$, we have 
		\[
		\left(\frac{-\Delta_{n}f^2}{7}\right)=\left(\frac{-\Delta_{n}}{7}\right),
		\]
		and hence combining Corollary \ref{cor:sc7oneH} (2), \eqref{eqn:Df^2}, and Lemma \ref{lem:Creval} gives 
		\[
		\operatorname{sc}_7(n)=\nu_{n} H(\Delta_n)\left(\sum_{d\mid 7} \mu(d)\left(\frac{-\Delta_n}{d}\right)\sigma\left(\frac{7}{d}\right)-1\right)\sum_{d\mid f} \mu(d)\left(\frac{-\Delta_n}{d}\right)\sigma\left(\frac{f}{d}\right).
		\]
		Noting that $7\nmid f$ and
		\begin{equation}\label{eqn:multiplicative}
		\sum_{d\mid f} \mu(d)\left(\frac{-\Delta_n}{d}\right)\sigma\left(\frac{f}{d}\right)
		\end{equation}
		is multiplicative, we obtain 
		\[
		\operatorname{sc}_7(n)=\nu_{n} H(\Delta_n)\left(\sum_{d\mid 7f} \mu(d)\left(\frac{-\Delta_n}{d}\right)\sigma\left(\frac{7f}{d}\right)-\sum_{d\mid f} \mu(d)\left(\frac{-\Delta_n}{d}\right)\sigma\left(\frac{f}{d}\right)\right).
		\]
		We then apply \eqref{eqn:Df^2} again and use Lemma \ref{lem:H7diff} to obtain Corollary \ref{cor:counting} in this case. This completes the base case $r_n=0$ of the induction.
		
		Let $r\geq 1$ be given and assume the inductive hypothesis that that Corollary \ref{cor:counting} holds for all $n$ with $r_n<r$. We then let $n$ be arbitrary with $r_n=r$ and show that Corollary \ref{cor:counting} holds for $n$. 
		By Corollary \ref{cor:sc7oneH} (4), we have
		\begin{equation}\label{eqn:toinduct}
		\operatorname{sc}_7(n)=7\operatorname{sc}_7\left(\frac{n+2}{7^2}-2\right).
		\end{equation}
		By the maximality of $r_n$, $7^{2r-1}\mid \frac{n+2}{7^2}$ but $7^{2r+1}\nmid \frac{n+2}{7^2}$, so $r_{\frac{n+2}{7^2}-2}=r-1<r$ and hence by induction we may plug Corollary \ref{cor:counting} into the right-hand side of \eqref{eqn:toinduct} to obtain
		\begin{equation}\label{eqn:inductstep}
		\operatorname{sc}_7(n)= 7\nu_{\frac{n+2}{7^2}-2}H_7\left(D_{\frac{n+2}{7^2}-2}\right).
		\end{equation}
		A straightforward calculation shows that
		\[
		\nu_{\frac{n+2}{7^2}-2}=\nu_n\quad \text{ and }\quad D_{\frac{n+2}{7^2}-2}=\frac{D_n}{7^2}
		\]
		and hence
		\eqref{eqn:inductstep} implies that
		\[
		\operatorname{sc}_7(n)=7\nu_{n}H_7\left(\frac{D_n}{7^2}\right).
		\]
		Hence Corollary \ref{cor:counting} in this case is equivalent to showing that 
		\begin{equation}\label{eqn:remarktoshow}
		H_7\left(D_n\right)= 7H_7\left(\frac{D_{n}}{7^2}\right).
		\end{equation}
		Plugging Lemma \ref{lem:H7diff} and then \eqref{eqn:Df^2} into both sides of \eqref{eqn:remarktoshow}, cancelling $H(\Delta_n)$, and again using the multiplicativity of \eqref{eqn:multiplicative}, one obtains that \eqref{eqn:remarktoshow} is equivalent to $C_{r+1,\Delta_n}=7C_{r,\Delta_n}$. 
		Since $r\geq 1$, we have $r+1\geq 2$, and Lemma \ref{lem:Creval} implies that $C_{r+1,\Delta_n}=7C_{r,\Delta_n}$, yielding Corollary \ref{cor:counting} for all $n\not\equiv 2\pmod{4}$.

		We finally consider the case $n\equiv 2\pmod{4}$. We choose $\ell$ maximally such that $n\equiv -2\pmod{2^{2\ell}}$. Lemma \ref{lem:div4} (1) implies that
		\[
		\operatorname{sc}_7(n)=\operatorname{sc}_7\left(\left({\frac{n+2}{2^{2\ell}}-2} +2\right)2^{2\ell}-2\right)=\operatorname{sc}_7\left({\frac{n+2}{2^{2\ell}}-2}\right).
		\]
		The choice of $\ell$ implies that $\frac{n+2}{2^{2\ell}}-2\not\equiv 2\pmod{4}$. We may therefore plug in Corollary \ref{cor:counting} and the definitions \eqref{eqn:Dndef} and \eqref{eqn:nudef} to conclude that 
		\[
		\operatorname{sc}_7\left(\frac{n+2}{2^{2\ell}}-2\right)= \nu_{\frac{n+2}{2^{2\ell}}-2} H_7\left(D_{\frac{n+2}{2^{2\ell}}-2}\right)=\nu_{n}H_7\left(D_n\right).\qedhere
		\]
	\end{proof}

\section{A combinatorial explanation of Corollary \ref{cor:counting}}\label{Section: combinatorics}
Here we provide a combinatorial explanation for Corollary \ref{cor:counting}. We use the theory of abaci, following the construction in \cite{ono19974}.

\subsection{Abaci, extended $t$-residue diagrams, and self-conjugate $t$-cores}

Given a partition $\Lambda = (\lambda_1, \lambda_2, \dots, \lambda_s)$ with $\lambda_1 \geq \lambda_2 \geq \dots \geq \lambda_s >0$ of a positive integer $n$ 
and a positive integer $t$, we next describe the \begin{it}$t$-abacus\end{it} associated to $\Lambda$. This 
consists of $s$ beads on $t$ rods constructed in the following way \cite{ono19974}. For every $1 \leq j \leq s$ define \emph{structure numbers} by
\begin{equation*}
B_j \coloneqq \lambda_j - j+s.
\end{equation*}
For each $B_j$ there are unique integers $(r_j,c_j)$ such that
\begin{equation*}
B_j = t(r_j-1) +c_j,
\end{equation*}
and $0 \leq c_j < t-1$. The \emph{abacus} for the partition $\Lambda$ is then formed by placing one bead for each $B_j$ in row $r_j$ and column $c_j$. 
The {\it extended $t$-residue diagram} associated to a $t$-core partition $\Lambda$ is constructed as follows 
(see \cite[page 3]{garvan1990cranks}).
 Label a cell in the $j$-th row and $k$-th column of the Ferrers--Young diagram of $\Lambda$ by $k-j \pmod{t}$. We also label the cells in column $0$ in the same way. A cell is called {\it exposed} if it is at the end of a row. The {\it region $r$} of the extended $t$-residue diagram of $\Lambda$ is the set of cells $(j,k)$ satisfying $t(r-1) \leq k-j < tr$. Then we define $n_j$ to be the maximum region of $\Lambda$ which contains an exposed cell labeled $j$. As noted in \cite{garvan1990cranks}, this is well-defined since column $0$ contains infinitely many exposed cells.

\begin{example}
	Let $t=4$ and construct the abacus and $4$-residue diagram for the partition $\Lambda = (3,2,1)$. 
	We begin with the abacus, computing the structure numbers $B_1 = 5$, $B_2 = 3$, and $B_3 = 1$. Then diagrammatically the abacus is 
	\begin{equation*}
	\begin{matrix}
	{}\vphantom{\begin{smallmatrix}a\\ a\end{smallmatrix}} & \text{\large{$0$}} &  \text{\large{$1$}} & \text{\large{$2$}} &\text{\large{$3$}}  \\
\text{\large{$1$}} \vphantom{\begin{matrix}a\\ a\end{matrix}}& {} & \text{\large{$B_3$}} & {} & \text{\large{$B_2$}}   \\
	\text{\large{$2$}}\vphantom{\begin{matrix}a\\ a\end{matrix}} & {} & \text{\large{$B_1$}} &  & 	
	\end{matrix}
	\end{equation*}
	The extended $4$-residue diagram of the partition is
	\begin{equation*}
	\begin{matrix}
	{} & \text{\large{$0$}}\vphantom{\begin{smallmatrix}a\\ a\end{smallmatrix}} & \text{\large{$1$}} & \text{\large{$2$}} & \text{\large{$3$}} \\
	\text{\large{$1$}}\vphantom{\begin{matrix}a\\ a\end{matrix}} & {}_3 & \text{\large $\bullet$}_0 & \text{\large $\bullet$}_1 & \text{\large $\bullet$}_2  \\
	\text{\large{$2$}}\vphantom{\begin{matrix}a\\ a\end{matrix}}  & {}_2 & \text{\large $\bullet$}_3 & \text{\large $\bullet$}_0 & {} 	\\
	\text{\large{$3$}}\vphantom{\begin{matrix}a\\ a\end{matrix}} & {}_1 & \text{\large $\bullet$}_2 & {} & 	{} \\
	\end{matrix}
	\end{equation*}
	Then the exposed cells in this diagram are $(1,3)$, $(2,2)$, and $(3,1)$. One may then determine the region of these cells in the prescribed fashion. For example, the exposed cell $(1,3)$ labeled by $2$ belongs to the region $1$, and hence $n_2 = 1$.
\end{example}

Using this construction, \cite[Theorem 4]{ono19974} reads as follows. Note that this is implicitly proven in \cite[Theorem 2.7.16]{JK81}.

\begin{theorem}\label{thm:t-coreabacus}
	Let $A$ be an abacus for a partition $\Lambda$, and let $m_j$ denote the number of beads in column $j$. Then $\Lambda$ is a $t$-core partition if and only if the $m_j$ beads in column $j$ are the beads in positions
	$
	(1,j), (2,j), \dots, (m_j,j).
	$
\end{theorem}

Furthermore, using extended $t$-residue diagrams, the authors of \cite{garvan1990cranks} showed the following result.
\begin{lemma}[Bijection 2 of \cite{garvan1990cranks}]\label{Lemma: Garvan size of lists}
	Let $P_t(n)$ be the set of $t$-core partitions of $n$. There is a bijection $P_t(n) \rightarrow \{ N \coloneqq [n_0, \dots, n_{t-1}] \colon n_j \in \Z, n_0 + \dots + n_{t-1} = 0 \}$ such that
	\begin{equation*}
	|\Lambda| = \frac{t|N|^2}{2} + B \cdot N, \hspace{20pt} B \coloneqq [0,1, \dots, t-1].
	\end{equation*}
	When computing the norm and dot-product, we consider $N,B$ as elements in $\Z^t$.
\end{lemma}

We call $N$ the {\it list associated} to the $t$-residue diagram. We now show a relationship between abaci and lists of a partition.

%

\begin{proposition}\label{Proposition: list to abacus}
	Let $N = [n_0,\dots, n_{t-1}]$ be the list associated to the extended $t$-residue diagram of a $t$-core partition $\Lambda$. Let $\ell + s = \alpha_\ell t + \beta_\ell$ with $0 \leq \beta_\ell \leq t-1$.
	Then $N$ also uniquely represents the abacus 
$
	( \dots,n_{t-1} +\alpha_{t-1}, n_0 + \alpha_0, n_1 + \alpha_1, \dots ),
$
	where $n_\ell+\alpha_\ell$ occurs in position $\beta_\ell$ of the abacus.
\end{proposition}

\begin{proof}
	The largest part $\lambda_1$ corresponds to the maximum region of the $t$-residue diagram, and also the lowest right-hand bead on the abacus. Let $m_1 \coloneqq \max\{n_0, \dots, n_{t-1}\}$ be achieved at $n_{\ell_1}$. Then $\lambda_1 = t(m_1-1) + \ell_1 + 1$. For the abacus, we correspondingly find $B_1 = \lambda_1 -1+s = t(m_1-1) +\ell_1 +s = t(m_1 + \alpha_1 -1) +\beta_1$, where $\ell_1 + s = \alpha_1 t + \beta_1$ with $0 \leq \beta_1 \leq t-1$. Hence we place a bead in the abacus at the slot $(m_1 +\alpha_1, \beta_1)$. Since this is a $t$-core partition, we also know that there are beads in all places above this slot. These beads correspond to other parts in the partition whose labels of exposed cells in the $t$-residue diagram are $\ell_1$ but where the exposed cells themselves lie in a lower region. Thus the $\beta_1$-th entry in the abacus takes value $m_1+\alpha_1$. 
	
	Then removing the element $n_{\ell_1}$ from the list we are left with $[n_0, \dots, n_{\ell_1  -1}, n_{\ell_1 +1}, \dots, n_{t-1}]$. We use the same technique as before, identifying $m_2 \coloneqq \max\{n_0, \dots, n_{\ell_1  -1}, n_{\ell_1 +1}, \dots, n_{t-1}\}$, achieved at $n_{\ell_2}$. We have $k-j \equiv \ell_2 \pmod{t}$ such that $t(m_2-1) \leq k-j < tm_2$, meaning that $\lambda_j = k = t(m_2-1)+\ell_2 +j$. Plugging this in to the formula for the structure numbers we find that $B_j = t(m_2 -1) +\ell_2 + s = t(m_2 + \alpha_2 - s) + \beta_2$, where $\ell_2 = \alpha_2 t + \beta_2$ with $0 \leq \beta < t$. Hence we place a bead in the abacus in the slot $(m_2+\alpha_2, \beta_2)$ and all other slots vertically above this, and so the $\beta_2$-nd entry in the abacus list is given by $m_2+\alpha_2$. This process continues for each entry of the list.
	
	If this process gives a non-positive value for the slots of the abacus in which beads are to be placed, we define the value in that column of the abacus list to be $0$ (it is seen that these values arise from the exposed cells in column $0$ of the extended $t$-residue diagram and hence are not a part of the partition). It is clear that the $\beta_\ell$ run through exactly a complete set of residues modulo $t$, and hence each column in the abacus is represented exactly once.
		It is easily seen that this process defines a unique abacus for each list $N$ (up to equivalency by Lemma \ref{Lemma: Abacus can have 0 first column}). The converse is also seen to hold.
	\end{proof}
	\begin{remark}
		If the resulting abacus $A$ that appears under an application of Proposition \ref{Proposition: list to abacus} has a non-zero first column, we may use Lemma \ref{Lemma: Abacus can have 0 first column} to rewrite $A$ as an equivalent abacus with a $0$ in the first place.
	\end{remark}

We use Proposition \ref{Proposition: list to abacus} to restrict the possible shapes of abaci associated to self-conjugate $t$-core partitions. 
\begin{lemma}\label{Lemma: shape of self-conj abaci}
	Let $t \geq 2$. With the notation defined as in Proposition \ref{Proposition: list to abacus}, an abacus is self-conjugate if and only if it is of the form
$
	(\dots, -n_1 + \alpha_1, -n_0+\alpha_0, n_0+\alpha_0, n_1+\alpha_1, \dots).
$
\end{lemma}
\begin{proof}
	The proof of \cite[Bijection 2]{garvan1990cranks} implies that the elements in the list $[n_0, \dots, n_{t-1}]$ associated to a self-conjugate partition satisfy the relations $n_{\ell} = -n_{t-\ell-1}$ for every $0 \leq \ell \leq t-1$. Combining this with Proposition \ref{Proposition: list to abacus} immediately yields the claim.
\end{proof}

\subsection{Self-conjugate $7$-cores}
	We now restrict our attention to abaci of self-conjugate $7$-cores. We require \cite[Lemma 1]{ono19974}, which allows us to form a system of canonical representatives for abaci associated to $7$-core partitions. Note that a similar result and following discussion holds for general $t$.
	\begin{lemma}\label{Lemma: Abacus can have 0 first column}
		The two abaci $A_1 = (m_0,m_1, \dots, m_6)$ and $A_2 = (m_{6}+1,m_0,\dots, m_{5})$ represent the same $7$-core partition.
	\end{lemma}
	
	Thus every $7$-core partition may be represented by an abacus of the form $(0,a,b,c,d,e,f)$. 
	Then in a similar fashion to Ono and Sze, we find that there is a one-to-one correspondence
	\begin{equation*}
	(0,a,b,c,d,e,f) \leftrightarrow \{ \text{all } 7\text{-core partitions} \},
	\end{equation*}
	where $a,b,c,d,e,$ and $f$ are non-negative integers. We thus assume that the first column in each abacus has no beads. 
	We next use Lemma \ref{Lemma: shape of self-conj abaci} to considerably reduce the number of abaci we need to consider.
	
	\begin{lemma}\label{Lemma: conditions on A,B..}
		Assume that $A = (0,a,b,c,d,e,f)$ is an abacus for a self-conjugate $7$-core partition and recall that $s = a+b+c+d+e+f$. Let $s \not\equiv 4 \pmod{7}$ and $r\in\N_0$. 
		\begin{enumerate}[wide, labelwidth=!, labelindent=0pt]
			\item[\rm (1)] Assume that $s=7r$. Then
$
			 f=2r,$ $
			 a+e=2r,$ $
			 b+d = 2r,$ $
			 c=r.$
			 
			\item[\rm (2)] Assume that $s=7r +1$. Then
$
			a=2r+1,$ $
			 b+f=2r,$ $
			 c+e=2r,$ $
			 d = r.$
	
			\item[\rm (3)] Assume that $s=7r+2$. Then
$
			b+a=2r+1,$ $
			 c=2r+1,$ $
			 d+f=2r,$ $
			 e = r.$

			\item[\rm (4)] Assume that $s=7r+3$. Then
$
			b+c = 2r+1,$ $
			a+d=2r+1,$ $
			e=2r+1,$ $
			 f = r.$

			\item[\rm (5)] Assume that $s=7r+5$. Then
$
			d+e=2r+1,$ $
			 c+f=2r+1,$ $
			 b=2r+2,$ $
			a= r+1.$
			
			\item[\rm (6)] Assume that $s=7r+6$. Then
$
			e+f=2r+1,$ $
			 d=2r+2,$ $
			a+c=2r+2,$ $
			b = r +1.$
			
		\end{enumerate}
	\end{lemma}
	\begin{proof}
		We prove (1). By Proposition \ref{Proposition: list to abacus} we see that $A$ corresponds to the list $[-r,a-r,b-r,c-r,d-r,e-r,f-r]$. Using Lemma \ref{Lemma: shape of self-conj abaci} and the fact that $s = 7r$, the conditions are easy to determine. The other cases follow in the same way.
	\end{proof}
	\begin{remarks*} \hfill
		\begin{enumerate}[wide, labelwidth=!, labelindent=0pt]
			
			\item It is clear how a similar result to Lemma \ref{Lemma: conditions on A,B..} may be obtained for all self-conjugate $t$-cores.
			
			\item 	The lack of the case $s \equiv 4 \pmod{7}$ in Lemma \ref{Lemma: conditions on A,B..} follows from the fact that there are no self-conjugate $(2t-1)$-core partitions with $s \equiv t \pmod{(2t-1)}$, which may be seen by inspecting the upper-left cell in the Ferrers--Young diagram of such a partition.
		\end{enumerate}
		
	\end{remarks*}

	Lemma \ref{Lemma: conditions on A,B..} shows that the abaci of self-conjugate $7$-core partitions naturally fall into one of the distinct families given in Table \ref{Table: acabi families}, enumerated with parameters $a,b,r \in \N_0$.

	\begin{center}
		\begin{table}[h!]
			\begin{tabular}{c c}
				{Type of Partition} & {Shape of Abaci} \\[5pt]
				I & $(0,a,b,r,2r-b,2r-a,2r)$ \\[5pt]
				II &  $	(0,2r+1,a,b,r,2r-b,2r-a)$\\[5pt] 
				III & $(0,a,2r+1-a,2r+1,b,r,2r-b)$ \\[5pt]
				IV & $(0,a,b,2r+1-b,2r+1-a,2r+1,r)$ \\[5pt]
				V &  $(0,r+1,2r+2,a,b,2r+1-b,2r+1-a)$\\[5pt]
				VI & $(0,a,r+1,2r+2-a,2r+2,b,2r+1-b)$ \\[10pt]
			\end{tabular} 
			\caption{\label{Table: acabi families} The different types of abaci for self-conjugate $7$-core partitions.}
		\end{table}
	\end{center}

	We relate the families of partitions to quadratic forms, with the relationship shown in the following proposition. For brevity, we write only triples without $\pm$ signs - it is clear that changing the sign on any entry preserves the result.

	\begin{proposition}\label{Prop: lists to quad forms}
		Let $n \in \N$ 
and $a,b,r\in\N_0$ be given.
		\begin{enumerate}[wide, labelwidth=!, labelindent=0pt]
			\item[\rm (1)] 
The Type I partition 
with parameters $a$, $b$, and $r$ 
is a partition of $n$ 
if and only if
			\begin{equation*}
			7n+14 =(7r+3)^2 + (7r+2-7a)^2 + ( 7r+1-7b)^2.
			\end{equation*}
			\item[\rm (2)] The Type II partition
with parameters $a$, $b$, and $r$ 
is a partition of $n$ 
if and only if 
			\begin{equation*}
			7n+14 = ( 7r+4 )^2 +( 7r+2-7a)^2 + (7r+1-7b )^2.
			\end{equation*}
			
			\item[\rm (3)] The Type III partition 
with parameters $a$, $b$, and $r$ 
is a partition of $n$ 
if and only if 
			\begin{equation*}
			7n+14 = (7r +5)^2 + (7r+4-7a )^2 + (7r+1-7b )^2.
			\end{equation*}
			
			\item[\rm (4)] The Type IV partition 
with parameters $a$, $b$, and $r$ 
 is a partition of $n$
if and only if 
			\begin{equation*}
			7n+14 = (7r+6  )^2 + (7r+5-7a  )^2 +  (7r+4-7b)^2.
			\end{equation*}
			
			\item[\rm (5)] The Type V partition 
with parameters $a$, $b$, and $r$ 
is a partition of $n$ 
if and only if
			\begin{equation*}
			7n+14 = (7r+8 )^2 + (7r+5-7a  )^2 + (7r+4-7b )^2.
			\end{equation*}

			\item[\rm (6)] The Type VI partition 
with parameters $a$, $b$, and $r$ 
 is a partition of $n$ 
if and only if
			\begin{equation*}
			7n+14 = (7r+9  )^2 + (7r+8-7a )^2 + (7r+4-7b )^2.
			\end{equation*}
		\end{enumerate}
	\end{proposition}
	\begin{proof}
		We only prove (1). Combining the definition with Proposition \ref{Proposition: list to abacus}, the Type I partition $\Lambda$ with parameters $a$, $b$, and $r$ has the associated list $[-r,a-r,b-r,0,r-b,r-a,r]$. By Lemma \ref{Lemma: Garvan size of lists}, we thus have
\[
n=|\Lambda|=7\left(r^2+(a-r)^2 +(b-r)^2\right) + (a-r)+2(b-r)+4(r-b)+5(r-a).
\]
Hence we 
see that 
		\begin{align*}
		7n+14 = & 49\left( r^2 + (a-r)^2 +(b-r)^2 \right) + 7\left(a-r + 2(b-r) +4(r-b) +5(r-a) +6r\right) + 14 \\
		=& 147r^2 + 49a^2 + 49b^2 + 84r - 98ar - 98br - 28a - 14b +14.
		\end{align*}
		This is exactly the expansion of
		\begin{equation*}
		(7r+3)^2 + (7r+2-7a)^2 + (7r+1-7b)^2.
		\end{equation*}
		The other cases follow in the same way, using the associated lists in Table \ref{Table: list families}.\qedhere
	\begin{center}
		\begin{table}[h!]
			\begin{tabular}{c c}
				{Type of Partition} & {Shape of Associated list} \\[5pt]
				I & $[-r,a-r,b-r,0,r-b,r-a,r]$ \\[5pt]
				II &  $[r+1,a-r,b-r,0,r-b,r-a,-r-1]$\\[5pt] 
				III & $[a-r,r+1,b-r,0,r-b,-r-1,r-a]$ \\[5pt]
				IV & $[a-r,b-r,r+1,0,-r-1,r-b,r-a]$ \\[5pt]
				V &  $[a-r,b-r,-r-1,0,r+1,r-b,r-a]$\\[5pt]
				VI & $[b-r,-r-1,a-r-1,0,r+1-a,r+1,r-b]$ \\[10pt]
			\end{tabular} 
			\caption{\label{Table: list families} The different types of associated lists for self-conjugate $7$-core partitions.}
		\end{table}
	\end{center}
\rm
	\end{proof}

	Proposition \ref{Prop: lists to quad forms} shows that for each self-conjugate $7$-core of $n$ there is 
	a representation of $7n+14 = x^2 + y^2 + z^2$ as the sum of three squares with none of $x,y,z$ divisible by $7$. Define 
	\begin{align*}
	J(7n+14) \coloneqq \{ (x,y,z) \in \Z^3 \colon & x^2+y^2+z^2 = 7n+14, \text{ and } x,y,z \not\equiv 0 \pmod{7} \}.
	\end{align*}
	Let $K(7n+14) \coloneqq J(7n+14)/ \sim$ where $(x,y,z) \sim (x',y',z')$ if $(x',y',z')$ is any permutation of the triple $(x,y,z)$, including minus signs i.e., $(-x,y,z) \sim (x,y,z)$. Then it is easy to see that we obtain the following corollary.
	
	\begin{corollary}\label{cor:sc7K}
		There is an isomorphism between self-conjugate $7$-core partitions and $K(7n+14)$.
	\end{corollary}
	
	\begin{remark}
Corollary \ref{cor:sc7K} gives a combinatorial explanation for Lemma \ref{lem:sc7Theta^3} (1). We then obtain an explanation of Corollary \ref{cor:counting} via the following exposition, using Gauss' bijective map from solutions of the equation $x^2+y^2+z^2 = 7n+14$ to primitive binary quadratic forms in certain class groups.
	\end{remark}
	
	We elucidate the case $n \equiv 0,1 \pmod{4}$ (a similar story holds for $n \equiv 3 \pmod{8}$). By 
Gauss's \cite[article 278]{10.2307/j.ctt1cc2mnd},
 for each representation of $7n+14$ as the sum of three squares there corresponds a primitive binary quadratic form of discriminant $-28n-56$. This correspondence is invariant under a pair of simultaneous sign changes on the triple $(x,y,z)$. Explicitly, the correspondence is given by the following. For $(x,y,z) \in J(7n+14)$ let $(m_0, m_1, m_2, n_0,n_1,n_2)$ be an integral solution to
	\begin{equation*}
	x=m_1n_2 - m_2n_1, \hspace{20pt} y = m_2n_0 - m_0n_2 \hspace{20pt} z = m_0n_1- m_1n_0,
	\end{equation*}
	where a solution is guaranteed by 
Gauss's \cite[article 279]{10.2307/j.ctt1cc2mnd}. 
	Then 
	\begin{equation*}
	(m_0u +n_0v)^2 + (m_1u +n_1v)^2 + (m_2u+n_2v)^2
	\end{equation*}
	is a form in $\operatorname{CL}(-28n-56)$. Further, this map is independent of $(m_0,m_1,m_2,n_0,n_1,n_2)$. Hence, similarly to \cite{ono19974}, we find a map $\phi$ taking self-conjugate $7$-cores to binary quadratic forms of discriminant $-28n-56$ given by
	
	\begin{equation}\label{Equation: defnintion phi}
	\phi \colon \Lambda \rightarrow A \rightarrow N \rightarrow (x,y,z) \rightarrow (m_0,m_1,m_2,n_0,n_1,n_2) \rightarrow \text{binary quadratic form}.
	\end{equation}
	
	We are now in a position to prove Theorem \ref{Theorem: main - quad forms}.
	
	\begin{proof}[Proof of Theorem \ref{Theorem: main - quad forms}]
		We first assume that $n \equiv 0,1 \pmod{4}$. It is well-known (see e.g. \cite[article 291]{10.2307/j.ctt1cc2mnd}) that we have $|\operatorname{CL}(-28n-56)| = 2^{r-1}k$, where $k$ is the number of classes per genus, and $2^{r-1}$ is the number of genera in $\operatorname{CL}(-28n-56)$. Fix $f_1, \dots, f_k$ to be representatives of the $k$ classes of the unique genus of $\operatorname{CL}(-28n-56)$ that $\phi$ maps onto. As in \cite{ono19974} we say that $(x,y,z)$ and $f_j$ are \emph{represented} by $(m_0,m_1,m_2,n_0,n_1,n_2)$ if
		\begin{equation*}
		x = m_1n_2 - m_2n_1, \hspace{20pt} y= m_2n_0 - m_0n_2, \hspace{20pt} z = m_0n_1 - m_1n_0,
		\end{equation*}
		and
		\begin{equation*}
		(m_0u + n_0v)^2 + (m_1u +n_1v)^2 + (m_2u +n_2v)^2 = f_j,
		\end{equation*}
		respectively. Let $\mathfrak{M}$ denote the set of all tuples $(m_0,m_1,m_2,n_0,n_1,n_2)$ that represent some pair $(x,y,z)$ and $f_j$. 
By Gauss's \cite[article 291]{10.2307/j.ctt1cc2mnd},
		 we have $|\mathfrak{M}| = 3\cdot2^{r+3}k$, and each $f_j$ is representable by $3\cdot 2^{r+3}$ elements in $\mathfrak{M}$. It is clear that all representatives $f_j$ have $(x,y,z) \in J(7n+14)$.
		 Note that the elements $(m_0,m_1,m_2,n_0,n_1,n_2)$ and $(-m_0,-m_1,-m_2, -n_0,-n_1,-n_2)$ both map to the same form in $K(7n+14)$, and there are no other such relations. Since each class in $K(7n+14)$ corresponds to $8\cdot 6$ different triples, we see that each element in $K(7n+14)$ has $\frac{3 \cdot 2^{r+3}}{8\cdot 6 \cdot 2} = 2^{r-2}$ different preimages. Hence the set of self-conjugate $7$-cores is a $2^{r-2}$-fold cover of this genus.
		To see that the genus is non-principal, we note as 
in \cite[Remark 3 ii)]{ono19974}
 that to be in the principal genus, one of $x,y,z$ would need to vanish. However, this is guaranteed to not happen by the congruence conditions on elements in $K(7n+14)$. The case where $n \equiv 3 \pmod{8}$ is similar. 
 
 Finally, for $n\equiv 2\pmod{4}$, one uses the simple fact that if the sum of three squares is congruent to $0$ modulo $4$, then all squares must be even. Iterating this eventually reduces it to one of the cases covered above or the $n\equiv 7\pmod{8}$ case. 
	\end{proof}

\section{Other $t$ and Conjecture \ref{Theorem: main intro}}\label{Section: other t and conjecture}
In this section we consider other values of $t$, proving Conjecture \ref{Theorem: main intro} in the cases $t\in\{2,3,5\}$ and offering partial results if $t>5$.
\subsection{The cases $t\in\lbrace 2,3 \rbrace$}\label{Sec: t=2,3}
With $\eta(\tau):=q^{\frac{1}{24}}\prod_{n\geq 1}(1-q^n)$ the usual {\it Dedekind eta-function},
 \cite[(3)]{ono19974} and \cite[Theorem 13]{alpoge2014self} 
 give the generating functions of $\operatorname{c}_2(n)$ and $\operatorname{sc}_3(n)$ as
\begin{align*}
\sum_{n \geq 1} \operatorname{c}_2(n)q^n  = q^{-\frac{1}{8}}\frac{\eta(2\tau)^2}{\eta(\tau)}, \qquad
\sum_{n \geq 1} \operatorname{sc}_3(n)q^n = q^{-\frac{1}{3}} \frac{\eta(2\tau)^2 \eta(3\tau) \eta(12\tau)}{\eta(\tau) \eta(4\tau) \eta(6\tau)}.
\end{align*}

These are modular forms of weight $\frac{1}{2}$ and levels $2$ and $12$, respectively. It is a classical fact that each is a lacunary series, i.e., that the asymptotic density of its non-zero coefficients is zero (for example, see the discussion 
after \cite[(2)]{MR1321575}). We immediately see that
\begin{equation*}
\operatorname{c}_2(n) = \begin{cases}
1 &\text{ if } n = \frac{j(j+1)}{2} \text{ for some } j \in \N, \\
0 &\text{ otherwise}.
\end{cases}
\end{equation*}
Furthermore, \cite[(7.4)]{garvan1990cranks} stated that 
\begin{equation*}
\operatorname{sc}_3(n) = \begin{cases}
1 &\text{ if } n = j(3j \pm 2) \text{ for some } j \in \N, \\
0 &\text{ otherwise}.
\end{cases}
\end{equation*}
From these, we immediately obtain the following corollary.
\begin{corollary}\label{Cor: sc3 = c2}
	For any $n\in \mathbb N$ that is both a triangular number and satisfies $n=j(3j \pm 2)$ for some $j \in \Z$ we have that $\operatorname{sc}_3(n) = \operatorname{c}_2(n) = 1$.
\end{corollary}
Clearly, there are progressions on which both $\operatorname{sc}_3(n)$ and $\operatorname{c}_2(n)$ trivially vanish. For example, we have $\operatorname{sc}_3(4n+3) = \operatorname{c}_2(3n+2) = 0$. For $t=3$ we simply observe the following corollary.
\begin{corollary}\label{Cor: t=3}
	There are no arithmetic progressions on which $\operatorname{c}_3$ and $\operatorname{sc}_5$ are integer multiples of one-another, even asymptotically.
\end{corollary}
\begin{proof}
	Comparing the explicit descriptions for $\operatorname{c}_3(n)$ and $\operatorname{sc}_5(n)$ given 
in \cite[Theorem 6]{hirschhorn2009elementary} and  \cite[Theorem 7]{garvan1990cranks}
 respectively immediately yields the claim.
\end{proof}

\subsection{The case $t =5$}
In \cite[Theorem 4]{garvan1990cranks},
 Garvan, Kim, and Stanton proved that 
\begin{equation*}
\operatorname{c}_5(n) = \sigma_5(n+1),
\end{equation*}
 where $\sigma_5(n) \coloneqq \sum_{d | n} (\frac{d}{5}) \frac nd $ denotes the usual twisted divisor sum. Furthermore, Alpoge provided an exact formula for $\operatorname{sc}_9(n)$ 
in \cite[Theorem 10]{alpoge2014self}:
\begin{align*}
27 \operatorname{sc}_9(n) = \begin{cases}
 \sigma(3n+10) + a_{3n+10} (36a) -a_{3n+10}(54a) - a_{3n+10}(108a)    & \text{ if } n \equiv 1,3 \pmod{4},\\
 \sigma(3n+10) + a_{3n+10} (36a) -3 a_{3n+10}(54a) - a_{3n+10}(108a)    & \text{ if } n \equiv 0 \pmod{4}, \\
 \sigma(k) + a_{3n+10} (36a) -3 a_{3n+10}(54a) - a_{3n+10}(108a)    & \text{ if } n \equiv 2 \pmod{4},
\end{cases}
\end{align*}
where $k$ is odd and is defined by $3n+10 = 2^e k$ where $e \in \N_0$ is maximal such that $2^e \mid (3n+10)$. Here, the $a_n(E)$ are the coefficients appearing in the Dirichlet series for the $L$-function of the elliptic curve $E$. The curve $36a$ is $y^2 = x^3+1$, the curve $54a$ is $y^2+xy = x^3-x^2+12x+8$, and the curve $108a$ is $y^2 = x^3 +4$. 

\begin{proposition}\label{Prop: t=5}
	There are no arithmetic progressions on which $27 \operatorname{sc}_9(n) $ and $\operatorname{c}_5(n)$ are asymptotically equal up to an integral multiplicative factor. 
\end{proposition}
\begin{proof}
Applying the Hasse--Weil bound for counting points on elliptic curves as in \cite[(13)]{alpoge2014self} and letting $n \rightarrow \infty$ we have, for $n \not\equiv 2 \pmod{4}$, that
	\begin{align*}
	\frac{27 \operatorname{sc}_9(n)}{\operatorname{c}_5(3n+9)} \sim \frac{\sigma(3n+10)}{\sigma_5(3n+10)},
	\end{align*}
	and for $n \equiv 2 \pmod{4}$
	\begin{align*}
	\frac{27 \operatorname{sc}_9(n)}{\operatorname{c}_5(n)} \sim \frac{\sigma(k)}{\sigma_5(n+1)}.
	\end{align*}
	
	It is then enough to show that $\sigma_5$ is never constant along arithmetic progressions, i.e., the limit is not constant. To see this, consider an arithmetic progression $n \equiv m \pmod{M}$. Let $\ell$ be any prime which does not divide $(3m+10)M$ and for which $\left(\frac{\ell}{5}\right)=-1$. For each prime $p\neq \ell$ that lies in the congruence class of the inverse of $\ell\pmod{3M}$ and is relatively prime to $5(3m+10)$ we may construct $n(p)=n_{\ell}(p)$ such that
\[
3n(p)+10=(3m+10)p\ell.
\]
Note that $3n(p)+10$ lies in the arithmetic progression. A straightforward calculation shows that if the limit exists, then 
\[
\lim_{p\to\infty}\frac{\sigma(3n(p)+10)}{\sigma_5(3n(p)+10)} = \pm \frac{1+\ell}{1-\ell}\frac{\sum_{d\mid (3m+10)} d}{\sum_{d\mid (3m+10)}\left(\frac{d}{5}\right) d}.
\]
Since $\ell$ is arbitrary and there are infinitely many choices of $\ell$ by Dirchlet's primes in arithmetic progressions theorem, this is a contradiction. \qedhere
\end{proof}

 \subsection{The case $t \geq 6$}\label{Section: t>5}
 Anderson showed 
in \cite[Theorem 2]{anderson2008asymptotic}
 that, for $t \geq 6$ and $n \rightarrow \infty$,
 \begin{equation}\label{Equation: asymps c_t}
 \operatorname{c}_t(n) = \frac{(2\pi)^{\frac{t-1}{2}} A_t(n) }{t^{\frac{t}{2}}
  \Gamma\left(\frac{t-1}{2}\right)}\left(n+\frac{t^2-1}{24}\right)^{\frac{t-3}{2}} +O\left(n^{\frac{t-1}{2}} \right),
 \end{equation}
 where 
 \begin{equation*}
 {A}_t(n) \coloneqq \sum_{\substack{k \geq 1 \\ \gcd(t,k)=1}} k^{\frac{1-t}{2}} \sum_{\substack{0 \leq h <k \\ \gcd(h,k)=1 }} e\left( - \frac{hn}{k}\right) \psi_{h,k}
 \end{equation*}
 for a certain $24k$-th root of unity $\psi_{h,k}$ independent of $n$. As Anderson remarked, it is possible to show that $0.05<A_t(n)<2.62$, although $A_t$ varies depending on both $t$ and $n$.
 
 In a similar vein, Alpoge showed
 in \cite[Theorem 3]{alpoge2014self}
 that, for $r \geq 10$ odd and $n \rightarrow \infty$, we have
 \begin{equation}\label{Equation: asymps for sc_t}
 \operatorname{sc}_r(n) = \frac{(2\pi)^{\frac{r-1}{4}}  \mathcal{A}_r(n)}{(2r)^{\frac{r-1}{4}} \Gamma \left(\frac{r-1}{4} \right) } \left( n+ \frac{r^2 -1}{24} \right)^{\frac{r-1}{4}-1} + O_r\left(n^{\frac{r-1}{8}}\right),
 \end{equation}
 where 
 \begin{equation*}
 \mathcal{A}_r(n) \coloneqq \sum_{\substack{\gcd(k,r) =1 \\ k \not\equiv 2 \pmod{4}}} (2,k)^{\frac{r-1}{2}} k^{\frac{1-r}{4}} \sum_{\substack{0 \leq h < k \\ \gcd(h,k) = 1}} e\left( - \frac{hn}{k}\right) \chi_{h,k}
 \end{equation*}
 with $\chi_{h,k}$ a particular $24$-th root of unity independent of $n$. 
Moreover, \cite[(86) and (87)]{alpoge2014self}
 imply that $ 0.14 < \mathcal{A}_r(n) < 1.86$.
 
 \begin{remark}
 	Inspecting the asymptotic behaviours given in \eqref{Equation: asymps c_t} and \eqref{Equation: asymps for sc_t}, it is clear that the only possibility of arithmetic progressions where the two asymptotics of $\operatorname{c}_t(n)$ and $\operatorname{sc}_r(n)$ are integer multiples of one another is $r=2t-1$. 
 \end{remark}

The following lemma provides partial results on Conjecture \ref{Theorem: main intro}.
 
 \begin{lemma}
 	For $t \geq 6$ and $t \not \equiv 1 \pmod{6}$ there are no arithmetic progressions on which $\operatorname{c}_t(n)$ and $\operatorname{sc}_{2t-1}(n)$ are integer multiples of one another.
 \end{lemma}
 \begin{proof}
 	Using equations \eqref{Equation: asymps c_t} and \eqref{Equation: asymps for sc_t} we find that, for $M_1,M_2,m_1,m_2 \in \N$,
 	\begin{align*}
 	\frac{\operatorname{c}_t(M_1n+m_1)}{\operatorname{sc}_{2t-1} (M_2n+m_2)} \sim \frac{(4t-2)^{\frac{t-1}{2}} A_t(M_1n+m_1)}{4^{\frac{t-3}{2}} t^{\frac{t}{2}} \mathcal{A}_{2t-1}(M_2n+m_2)} \frac{\left( 24( M_1n+m_1) + t^2 - 1\right)^{\frac{t-3}{2}}}{\left(6 (M_2 n+m_2) + t^2 - t\right)^{\frac{t-3}{2}}} ,
 	\end{align*}
 	as $n \rightarrow \infty$. Furthermore, for the two growing powers of $n$ to be equal and cancel on arithmetic progressions, it is not difficult to see that we must also have that $t \equiv 1 \pmod{6}$. 	
 \end{proof}

 To prove Conjecture \ref{Theorem: main intro} it remains to consider the cases where $t \equiv 1 \pmod{6}$. We easily find that for the powers of $n$ to be equal we must have
 \begin{equation*}
 M_2 = 4M_1, \qquad m_2 = 4m_1+\frac{t^2-1}{6}.
 \end{equation*}
 It would therefore suffice to show that
 \begin{equation*}
  \frac{ (4t-2)^{\frac{t-1}{2}}  A_t(M_1n+m_1)}{4^{\frac{t-3}{2}} t^{\frac{t}{2}} \mathcal{A}_{2t-1}\left(4M_1n+4m_1+\frac{t^2-1}{6}\right)} 
 \end{equation*}
 is never constant as $n$ runs. However, this appears to be a difficult problem, and we leave Conjecture \ref{Theorem: main intro} open.

\bibliographystyle{amsplain}

\end{document}